\documentclass[11pt]{article}

\usepackage{amsmath}
\usepackage{amssymb}
\usepackage{amsthm}
\usepackage[english]{babel}
\usepackage[utf8]{inputenc}
\usepackage[left=3cm,right=3cm,top=3cm,bottom=3cm]{geometry}
\usepackage{setspace}
\allowdisplaybreaks[3]

\newtheorem{theorem}{Theorem}[section]

\newtheorem{lemma}[theorem]{Lemma}
\newtheorem{proposition}[theorem]{Proposition}
\newtheorem{example}[theorem]{Example}
\theoremstyle{definition}
\newtheorem{definition}[theorem]{Definition}
\theoremstyle{remark}
\newtheorem{remark}[theorem]{Remark}
\numberwithin{equation}{section}

\newcommand{\AKMSbracket}[1]{\left[#1\right]}
\newcommand{\AKMSpara}[1]{\left(#1\right)}

\begin{document}
\author{Abdennour Kitouni\footnote{e-mail: abdennour.kitouni@uha.fr}\\ Université de Haute-Alsace,\\ Mulhouse France
\and Abdenacer Makhlouf\footnote{email: abdenacer.makhlouf@uha.fr} \\ Université de Haute-Alsace, \\Mulhouse France
}
\title{On Structure and Central Extensions of $(n+1)$-Lie Algebras Induced by $n$-Lie Algebras}
\date{}

\maketitle

\abstract{
The purpose of this paper is to investigate $(n+1)$-Lie algebras induced by $n$-Lie algebras and  trace maps. We highlight a comparison of   their structure  properties (solvability, nilpotency) and  the cohomology groups as well as central extensions. Moreover, we provide for dimensions $n$, $n+1$ and $n+2$, the classification of $n$-Lie algebras which are induced by $(n-1)$-Lie algebras.
}
\section*{Introduction}
In this paper, we investigate $(n+1)$-Lie algebras constructed from $n$-Lie algebras and generalized trace maps. They are called $(n+1)$-Lie algebras induced by $n$-Lie algebras. 

Ternary Lie algebras   appeared first in Nambu's generalization of Hamiltonian mechanics \cite{Nambu:GenHD} which uses a generalization of Poisson algebras with a ternary bracket. The algebraic formulation of Nambu mechanics is due to Takhtajan while the structure of $n$-Lie algebra was studied by Filippov \cite{Filippov:nLie} then completed by Kasymov in \cite{Kasymov:nLie}, where solvability and nilpotency properties were studied. 

The cohomology of $n$-Lie algebras, generalizing the Chevalley-Eilenberg Lie algebras cohomology, was first introduced by Takhtajan \cite{Takhtajan:cohomology} in its simplest form, later a complex adapted to the study of formal deformations was introduced by Gautheron \cite{Gautheron:Rem}, then reformulated by Daletskii and Takhtajan \cite{Dal_Takh} using the notion of base Leibniz algebra of an $n$-Lie algebra.

In \cite{almy:quantnambu}, the authors introduced a realization of the quantum Nambu bracket in terms of matrices (using the commutator and the trace of matrices). This construction was generalized in \cite{ams:ternary} to the case of any Lie algebra where the commutator is replaced by the Lie bracket, and the matrix trace is replaced by linear forms having similar properties. Therefore one obtains  ternary brackets which define  $3$-Lie algebras, called $3$-Lie algebras induced by Lie algebras. This construction was generalized to the case of $n$-Lie algebras in \cite{ams:n}. 

The aim of this paper is to  study  the relationships between the structure properties (solvability, nilpotency,...), central extensions, and the cohomology of an $n$-Lie algebra and its induced $(n+1)$-Lie algebra.
In Section \ref{AKMS:n-Lie}, we recall the basics about $n$-Lie algebras, $n$-Lie algebras cohomology and the construction of $(n+1)$-Lie algebras induced by $n$-Lie algebras. In Section \ref{AKMS:Structure} we discuss structure properties of $(n+1)$-Lie algebras induced by $n$-Lie algebras, likewise subalgebras, ideals, solvability and nilpotency. It is shown that the induced $(n+1)$-Lie algebra is always solvable and a  nilpotent $n$-Lie algebra gives rise to an  induced $(n+1)$-Lie algebras which is nilpotent as well.  In Section \ref{AKMS:Extension}, we deal with central extensions. We study the relationships between a central extension of an $n$-Lie algebra and those of an $(n+1)$-Lie algebra it induces. Moreover we study the corresponding cohomology. In Section \ref{AKMS:Cohomology},  we compare $1$-cocycles and $2$-cocycles, with respect to the adjoint cohomology and scalar cohomology,  of an $n$-Lie algebra and the induced $(n+1)$-Lie algebra.  In Section \ref{AKMS:Classifications}, we provide a sufficient condition to recognize $(n+1)$-Lie algebras that are induced by  $n$-Lie algebras.  We determine all $n$-Lie algebras induced by $(n-1)$-Lie algebras up to dimension $n+2$. This classification is based on classifications given in \cite{Bai:nLie:n+2} and \cite{Filippov:nLie}. Hence, we provide a list of all  possible induced $3$-Lie algebras by $n$-dimensional Lie algebras with $n\leq 4$.

\section{Preliminaries}\label{AKMS:n-Lie}
In this paper, all vector spaces are over a field $\mathbb{K}$ of characteristic $0$.
The structure of $n$-Lie algebra was introduced by Filippov \cite{Filippov:nLie}, then investigated deeply by Kasymov \cite{Kasymov:nLie}. Let us recall of some basic definitions.
\subsection{Definitions and first results}
  \begin{definition}
    An $n$-Lie algebra $\AKMSpara{A,\AKMSbracket{\cdot,...,\cdot}}$ is a vector space $A$ together with a skew-symmetric $n$-linear map $\AKMSbracket{\cdot,...,\cdot} : A^n \to A$ such that:
    \begin{equation}
 \AKMSbracket{x_1,...,x_{n-1},\AKMSbracket{y_1,...,y_n}} = \sum_{i=1}^n \AKMSbracket{y_1,...,\AKMSbracket{x_1,...,x_{n-1},y_i},...,y_n}, \label{AKMS-FI}
    \end{equation}
   for all $x_1,...,x_{n-1},y_1,...,y_n \in A$. This condition is called the fundamental identity or Filippov identity. For $n=2$, identity (\ref{AKMS-FI}) becomes the Jacobi identity and we get the definition of a Lie algebra.
  \end{definition} 
  Recall that the $n$-linear bracket is said to be skew-symmetric if for all $\sigma$ in the permutation group $S_n$ and $x_1,\cdots,x_n\in A$, we have $[x_{\sigma(1)},\cdots,x_{\sigma(n)}]=sgn(\sigma)[x_{1},\cdots,x_{n}]$,  where $sgn(\sigma)$ is the signature of $\sigma$.

\begin{definition}[Fundamental object]
Let $(A,\AKMSbracket{\cdot,...,\cdot})$ be an $n$-Lie algebra. A fundamental object $X$ is defined by $(n-1)$ elements of $A$, $x_1,...,x_{n-1}$, on which it is skew-symmetric, that is $X \in \wedge^{n-1}A$.

We define the action of fundamental objects on $A$ by : \[ \forall X \in \wedge^{n-1}A, \forall z \in A : X \cdot z = ad_X(z) = \AKMSbracket{x_1,...,x_{n-1},z}. \]

The multiplication (composition) of two fundamental objects $X, Y$ is given by : \[ X \cdot Y = \sum_{i=1}^{n-1} \left( y_1,...,X \cdot y_i , ..., y_{n-1} \right). \]
\end{definition}

\begin{proposition}
The multiplication of fundamental objects satisfies:
\begin{itemize}
\item \[ X \cdot (Y \cdot Z) = (X \cdot Y) \cdot Z + Y \cdot (X \cdot Z) \tag{Leibniz Rule}, \]
\item \[ ad_{X\cdot Y} = - ad_{Y\cdot X}, \]
\item \[ ad_{X \cdot Y} = ad_X\circ ad_Y - ad_Y \circ ad_X. \]
\end{itemize}
\end{proposition}

\begin{remark}
Using the notion of fundamental object, we can rewrite the fundamental identity as:
\[ X\cdot (Y\cdot z) = (X\cdot Y)\cdot z + Y \cdot (X\cdot z), \qquad \forall X,Y \in \wedge^{n-1}A, \forall z \in A. \]
\end{remark}
\begin{definition}
Let $(A,\AKMSbracket{\cdot,...,\cdot})$ be an $n$-Lie algebra, and $I$ a subspace of $A$. We say that $I$ is an ideal of $A$ if, for all $i\in I, x_1,...,x_{n-1}\in A$, it holds that $\AKMSbracket{i,x_1,...,x_{n-1}}\in I$.
\end{definition}

\begin{lemma}
Let $(A,\AKMSbracket{\cdot,...,\cdot})$ be an $n$-Lie algebra and $I_1,....,I_n$ be ideals of $A$, then $I=\AKMSbracket{I_1,...,I_n}$ is an ideal of $A$.
\end{lemma}

\begin{definition}
Let $(A,\AKMSbracket{\cdot,...,\cdot})$ be an $n$-Lie algebra and $I$ be an ideal of $A$. Define the derived series of $I$ by:
\[D^0(I)=I \text{ and } D^{p+1}(I)=\AKMSbracket{D^p(I),...,D^p(I)},\]
and the central descending series of $I$ by:
\[C^0(I)=I \text{ and } C^{p+1}(I)=\AKMSbracket{C^p(I),I,...,I}.\]
\end{definition}

\begin{definition}
Let $(A,\AKMSbracket{\cdot,...,\cdot})$ be an $n$-Lie algebra, and $I$ an ideal of $A$. The ideal $I$ is said to be solvable if there exists $p \in \mathbb{N}$ such that $D^p(I)=\{0\}$. It is said to be nilpotent if there exists $p \in \mathbb{N}$ such that $C^p(I)=\{0\}$.
\end{definition}

\begin{definition}
Let $(A,\AKMSbracket{\cdot,...,\cdot})$ be an $n$-Lie algebra. The center of $A$, denoted by $Z(A)$, is an ideal of $A$ defined by:
\[ Z(A) = \left\{ z \in A : \AKMSbracket{z,x_1,...,x_{n-1}} = 0, \forall x_1,...,x_{n-1} \in A  \right\}. \]
\end{definition}

\begin{definition}
An $n$-Lie algebra $(A,\AKMSbracket{\cdot,...,\cdot})$ is said to be simple if $\left. D^1(A) \neq \{ 0 \} \right.$ and if it has no ideals other than $\{0\}$ and $A$. A direct sum of simple $n$-Lie algebras is said to be semi-simple.
\end{definition}

\subsection{Cohomology of $n$-Lie algebras }
Now, let us recall the main definitions for $n$-Lie algebras cohomology, for references see \cite{Dal_Takh}, \cite{aip:review}, \cite{Gautheron:Rem} and \cite{Takhtajan:cohomology}.

\begin{definition}
Let $(A,[\cdot ,..., \cdot])$ be an $n$-Lie algebra. An $A$-valued $p$-cochain is a linear map $\psi : (\wedge^{n-1} A)^{\otimes p-1} \wedge A \to A$.

 The coboundary operator for the adjoint action is given by :
\begin{align*}
d^p\psi(X_1,...,X_p,z) &= \sum_{j=1}^p \sum_{k=j+1}^p (-1)^j \psi \left( X_1,...,\widehat{X_j},..., X_j \cdot X_k, ..., X_p,z \right)\\
&+ \sum_{j=1}^p (-1)^j \psi \left( X_1,...,\widehat{X_j},...,X_p,X_j \cdot z \right) \\
&+ \sum_{j=1}^p (-1)^{j+1} X_j \cdot \psi \left( X_1,...,\widehat{X_j},...,X_p, z \right) \\
&+ (-1)^{p-1} \left( \psi (X_1,...,X_{p-1}, \quad) \cdot X_p \right) \cdot z,
\end{align*}
where \[ \psi (X_1,...,X_{p-1}, \quad) \cdot X_p = \sum_{i=1}^{n-1} \left(x_p^1,...,x_p^{i-1},\psi (X_1,...,X_{p-1}, ...,x_p^{n-1}\right). \]
\end{definition}

\begin{definition}
Let $(A,[\cdot,...,\cdot])$ be an $n$-Lie algebra. A $\mathbb{K}$-valued $p$-cochain is a linear map $\psi : (\wedge^{n-1} A)^{\otimes p-1} \wedge A \to \mathbb{K}$.

 The coboundary operator for the trivial action is given by :
\begin{align*}
d^p\psi(X_1,...,X_p,z)&= \sum_{j=1}^p \sum_{k=j+1}^p (-1)^j \psi \left( X_1,...,\widehat{X_j},..., X_j \cdot X_k, ..., X_p,z \right)\\
&+ \sum_{j=1}^p (-1)^j \psi \left( X_1,...,\widehat{X_j},...,X_p,X_j \cdot z \right). 
\end{align*}

\end{definition}

The elements of $Z^p(A,A) = \ker d^p$ (resp. $Z^p(A,\mathbb{K})$) are called $p$-cocycles, those of $B^p(A,A)= \operatorname{Im} d^{p-1}$ (resp. $B^p(A,\mathbb{K})$) are called coboundaries. The quotient $H^p=\frac{Z^p}{B^p}$ is the $p$-th cohomology group. We sometimes add in subscript the representation used in the cohomology complex, for example $Z_{ad}^p(A,A)$ denotes the set of $p$-cocycle for the adjoint cohomology and $Z_{0}^p(A,\mathbb{K})$  denotes the set of $p$-cocycle for the scalar cohomology.

In particular, the elements of $Z^1(A,A)$ are the derivations. Recall that a derivation of an $n$-Lie algebra is a linear map $f : A \to A$ satisfying:
\[ f\AKMSpara{\AKMSbracket{x_1,...,x_n}} = \sum_{i=1}^n \AKMSbracket{x_1,...,f(x_i),...,x_n},\forall x_1,...,x_n \in A. \]

\subsection{$(n+1)$-Lie algebras induced by $n$-Lie algebras}

In \cite{ams:ternary} and \cite{ams:n}, the authors introduced a construction of a $3$-Lie algebra from a Lie algebra, and more generally an $(n+1)$-Lie algebra from an $n$-Lie algebra were introduced. 

\begin{definition}
  Let $\phi:A^n\to A$ be an $n$-linear map and $\tau:A\rightarrow \mathbb{K}$ be a linear map. Define $\phi_\tau:A^{n+1}\to A$ by:
  \begin{align}
    \phi_\tau(x_1,...,x_{n+1}) = \sum_{k=1}^{n+1}(-1)^k\tau(x_k)\phi(x_1,...,\hat{x}_k,...,x_{n+1}),
  \end{align}
where the hat over $\hat{x}_k$ on the right hand side means that $x_{k}$ is excluded, that is 
 $\phi$ is calculated on $(x_1,\ldots, x_{k-1}, x_{k+1}, ... , x_{n+1})$. 
\end{definition}
We will not be concerned with just any linear map $\tau$, but rather maps that have a generalized trace property. Namely:

\begin{definition}
  For $\phi:A^n\to A$, we call a linear map $\tau:A \to \mathbb{K}$ a \emph{$\phi$-trace (or trace)} if $\tau\AKMSpara{\phi(x_1,\ldots,x_n)}=0$ for $x_1,\ldots,x_n\in A$.
\end{definition}

\begin{lemma}
  Let $\phi:A^n\to A$ be a skew-symmetric $n$-linear map and $\tau$ a linear map $A\to\mathbb{K}$. Then $\phi_\tau$ is an $(n+1)$-linear skew-symmetric map. Furthermore, if $\tau$ is a $\phi$-trace then $\tau$ is a $\phi_\tau$-trace.
\end{lemma}

\begin{theorem} \cite{ams:n}
Let $(A,\phi)$ be an $n$-Lie algebra and $\tau$ a $\phi$-trace, then $(A,\phi_\tau)$ is an $(n+1)$-Lie algebra. We shall say that $(A,\phi_\tau)$ is induced by $(A,\phi)$. We refer to $A$ when considering the given $n$-Lie algebra and $A_\tau$ when considering the induced $(n+1)$-Lie algebra.
\end{theorem}


\section{Structure of $(n+1)$-Lie Algebras induced by $n$-Lie Algebras}\label{AKMS:Structure}
In this section, we discuss some  structure properties  of $(n+1)$-Lie algebras induced by $n$-Lie algebras (subalgebras, ideals, solvability and nilpotency). They generalize to $n$-ary case the results obtained in \cite{akms:ternary}  and independently in \cite{Bai:rlz3} for ternary algebras.

Let $(A,\AKMSbracket{\cdot,...,\cdot})$ be an $n$-Lie algebra, $\tau$ a trace and $(A,\AKMSbracket{\cdot,...,\cdot}_\tau)$ the induced $(n+1)$-Lie algebra.
\begin{proposition}
If $B$ is a subalgebra of $A$ then $B$ is also a subalgebra of $A_\tau$.
\end{proposition}
\begin{proof}
Let $B$ be a subalgebra of $(A,\AKMSbracket{\cdot,...,\cdot})$ and $x_1,...,x_{n+1} \in B$:
\[ \AKMSbracket{x_1,...,x_{n+1}}_\tau = \sum_{i=1}^{n+1} (-1)^{i-1}\tau(x_i)\AKMSbracket{x_1,...,\hat{x}_{i},...,x_{n+1}},\]
which is a linear combination of elements of $B$ and then belongs to $B$.
\end{proof}

\begin{proposition}
Let $J$ be an ideal of  $A$. Then $J$ is an ideal of $A_\tau$ if and only if
\[  \AKMSbracket{A,...,A} \subseteq J \text{ or } J \subseteq \ker \tau. \]
\end{proposition}
\begin{proof}
Let $J$ be an ideal of $A$, and let $j \in J$ and $x_1,...,x_n \in A$, then we have:
\[ \AKMSbracket{x_1,...,x_n,j}_\tau = \sum_{i=1}^n (-1)^{i-1} \tau(x_i) \AKMSbracket{x_1,...,\hat{x}_i,...,x_n,j} + (-1)^n \tau(j) \AKMSbracket{x_1,...,x_n}. \]
We have $\sum_{i=1}^n (-1)^{i-1} \tau(x_i) \AKMSbracket{x_1,...,\hat{x}_i,...,x_n,j} \in J$, then, to obtain $\left. \AKMSbracket{x_1,...,x_n,j}_\tau \in J \right.$ it is necessary and sufficient to have $\tau(j) \AKMSbracket{x_1,...,x_n} \in J$, which is equivalent to $\tau(j) = 0$ or $\AKMSbracket{x_1,...,x_n} \in J$. 
\end{proof}

\begin{theorem} \label{AKMSsolv2}
Let $(A,\AKMSbracket{\cdot,...,\cdot})$ be an $n$-Lie algebra, $\tau$ a $\AKMSbracket{\cdot,...,\cdot}$-trace and $(A,\AKMSbracket{\cdot,...,\cdot}_\tau)$ the induced $(n+1)$-Lie algebra. The $(n+1)$-Lie algebra $(A,\AKMSbracket{\cdot,...,\cdot}_\tau)$ is solvable, more precisely $D^2(A_\tau) = 0$ i.e. $\left( D^1(A_\tau)=\AKMSbracket{A,...,A}_\tau,\AKMSbracket{\cdot,...,\cdot}_\tau\right)$ is abelian.
\end{theorem}
\begin{proof}
Let $x_1,...,x_{n+1}\in \AKMSbracket{A,...,A}_\tau$, $x_i=\AKMSbracket{x_i^1,...,x_i^{n+1}}_\tau$,  $\forall 1 \leq i \leq n+1$ , then:
\begin{align*}
&\AKMSbracket{x_1,...,x_{n+1}}_\tau \\ &= \sum_{i=1}^{n+1} \tau\AKMSpara{\AKMSbracket{x_i^1,...,x_i^{n+1}}_\tau}\AKMSbracket{\AKMSbracket{x_1^1,...,x_1^{n+1}}_\tau,...,\widehat{\AKMSbracket{x_i^1,...,x_i^{n+1}}_\tau},...\AKMSbracket{x_{n+1}^1,...,x_{n+1}^{n+1}}_\tau}\\ 
 &= 0, 
\end{align*}
because $\tau\left([\cdot, ... ,\cdot]_\tau \right)=0$.

\end{proof}

\begin{remark}[\cite{Filippov:nLie}]
Let $(A,\AKMSbracket{\cdot,...,\cdot})$ be an $n$-Lie algebra. If we fix $a \in A$, the bracket $ \AKMSbracket{\cdot,...,\cdot}_a = \AKMSbracket{a,...,\cdot} $ is skew-symmetric and satisfies the fundamental identity. Indeed, we have, for \\$x_1,...,x_{n-2},y_1,...,y_{n-1} \in A$:
\begin{align*}
\AKMSbracket{x_1,...,x_{n-2},\AKMSbracket{y_1,...,y_{n-1}}_a}_a &= \AKMSbracket{a,x_1,...,x_{n-2},\AKMSbracket{a,y_1,...,y_{n-1}}}\\
&= \AKMSbracket{\AKMSbracket{a,x_1,...,x_{n-2},a},y_1,...,y_{n-1}} \\
&+ \sum_{i=1}^{n-1} \AKMSbracket{a,y_1,...,\AKMSbracket{a,x_1,...x_{n-2},y_i},...,y_{n-1}} \\
&= \sum_{i=1}^{n-1} \AKMSbracket{y_1,...,\AKMSbracket{x_1,...x_{n-2},y_i}_a,...,y_{n-1}}_a.
\end{align*}

\end{remark}

\begin{proposition}
Let $(A,\AKMSbracket{\cdot,...,\cdot})$ be an $n$-Lie algebra, $\tau$ a $\AKMSbracket{\cdot,...,\cdot}$-trace and $(A,\AKMSbracket{\cdot,...,\cdot}_\tau)$ the induced $(n+1)$-Lie algebra. Let $c \in Z(A)$, if $\tau(c) = 0$ then $c \in Z(A_\tau)$. Moreover, if $A$ is not abelian then $\tau(c) = 0$ if and only if $c \in Z(A_\tau)$.
\end{proposition}

\begin{proof}
Let $c \in Z(A)$ and $x_1,...,x_n \in A$:
\begin{align*}
\AKMSbracket{x_1,...,x_n,c}_\tau &= \sum_{i=1}^n (-1)^{i-1} \tau(x_i)\AKMSbracket{x_1,...,\widehat{x_i},...,x_n,c} + (-1)^n \tau(c) \AKMSbracket{x_1,...,x_n}\\
&= (-1)^n \tau(c) \AKMSbracket{x_1,...,x_n}.
\end{align*}
If $\tau(c)=0$ then $c \in Z(A_\tau)$.\\ Conversely, if $c \in Z(A_\tau)$ and $A$ is not abelian, then $\tau(c)=0$. 
\end{proof}

\begin{proposition}
Let $(A,\AKMSbracket{\cdot,...,\cdot})$ be a non-abelian $n$-Lie algebra, $\tau$ a $\AKMSbracket{\cdot,...,\cdot}$-trace and $(A,\AKMSbracket{\cdot,...,\cdot}_\tau)$ the induced $(n+1)$-Lie algebra. If $\tau\AKMSpara{Z(A)} \neq \{0\}$ then $A_\tau$ is not abelian.
\end{proposition}

\begin{proof}
Let $x_1,...,x_n \in A$ such that $\AKMSbracket{x_1,...,x_n} \neq 0$ and $c \in Z(A)$ such that $\tau(c) \neq 0$ then we have:
\begin{align*}
\AKMSbracket{x_1,...,x_n,c}_\tau &= \sum_{i=1}^n (-1)^{i-1} \tau(x_i)\AKMSbracket{x_1,...,\widehat{x_i},...,x_n,c} + (-1)^n \tau(c) \AKMSbracket{x_1,...,x_n}\\
&= (-1)^n \tau(c) \AKMSbracket{x_1,...,x_n} \neq 0,
\end{align*}
which means that $A_\tau$ is not abelian.
\end{proof}

\begin{proposition}
Let $\AKMSpara{A,\AKMSbracket{\cdot,...,\cdot}}$ be an $n$-Lie algebra, $\tau$ be a trace, $\AKMSpara{A,\AKMSbracket{\cdot,...,\cdot}_\tau}$ the induced algebra,  $\AKMSpara{C^p(A)}_p$ be the central descending series of $\AKMSpara{A,\AKMSbracket{\cdot,...,\cdot}}$, and $\AKMSpara{C^p(A_\tau)}_p$ be the central descending series of $\AKMSpara{A,\AKMSbracket{\cdot,...,\cdot}_\tau}$. Then we have 
\[ C^p(A_\tau) \subset C^p(A), \forall p \in \mathbb{N}. \]
If there exists $u \in A$ such that $\AKMSbracket{u,x_1,...,x_n}_\tau = \AKMSbracket{x_1,...,x_n}, \forall x_1,...,x_n \in A$ then:
\[ C^p(A_\tau) = C^p(A), \forall p \in \mathbb{N}. \]
\end{proposition}
\begin{proof}
We proceed by induction over $p \in \mathbb{N}$. The case of $p=0$ is trivial, for $p=1$ we have:
\[\forall x = \AKMSbracket{x_1,...,x_{n+1}}_\tau \in C^1(A_\tau),\  x=\sum_{i=1}^{n+1} (-1)^{i-1} \tau(x_i) \AKMSbracket{x_1,...,\hat{x}_i,...,x_{n+1}}, \] which is a linear combination of elements of $C^1(A)$ and then is an element of $C^1(A)$.
Suppose now that there exists $u \in A$ such that $\AKMSbracket{u,x_1,...,x_n}_\tau = \AKMSbracket{x_1,...,x_n}, \forall x_1,...,x_n \in A$. Then for $x = \AKMSbracket{x_1,...,x_n} \in C^1(A)$, $x=\AKMSbracket{u,x_1,...,x_n}_\tau$ and hence it  is an element of $C^1(A_\tau)$.

Now, we suppose this proposition  true for some $p \in \mathbb{N}$, and let $x \in C^{p+1}(A_\tau)$. Then $x=\AKMSbracket{a,x_1,...,x_n}_\tau$ with $x_1,...,x_n \in A$ and $a \in C^{p}(A_\tau)$,
\[x=\AKMSbracket{a,x_1,...,x_n}_\tau = \sum_{i=1}^{n} (-1)^{i} \tau(x_i) \AKMSbracket{a,x_1,...,\hat{x}_i,...,x_n}, \qquad (\tau(a)=0)\]
which is an element of $C^{p+1}(A)$ because $a \in C^{p}(A_\tau) \subset C^p(A)$.
Assume there exists $u \in A$ such that $\AKMSbracket{u,x_1,...,x_n}_\tau = \AKMSbracket{x_1,...,x_n}, \forall x_1,...,x_n \in A$. Then if $x \in C^{p+1}(A)$ we have $x = \AKMSbracket{a,x_1,...,x_{n-1}}$ with $a \in C^{p}(A)$ and $x_1,...,x_{n-1} \in A$. Therefore: \[ x = \AKMSbracket{a,x_1,...,x_{n-1}} = \AKMSbracket{u,a,x_1,...,x_{n-1}}_\tau = (-1)^n \AKMSbracket{a,x_1,...,x_{n-1},i}_\tau \in C^{p+1}(A_\tau).\]  
\end{proof}

\begin{remark} \label{AKMS-D3subD}
It also results from the preceding proposition that \[D^1(A_\tau) = \AKMSbracket{A,...,A}_\tau \subset D^1(A) = \AKMSbracket{A,...,A},\] and if there exists $u \in A$ such that \[\AKMSbracket{u,x_1,...,x_n}_\tau = \AKMSbracket{x_1,...,x_n}, \forall x_1,...,x_n \in A,\] then $D^1 (A_\tau) = D^1(A)$. For the rest of the derived series, we have obviously the first inclusion by Theorem \ref{AKMSsolv2}, which states also that all induced algebras are solvable.
\end{remark}

\begin{theorem}
Let $\AKMSpara{A,\AKMSbracket{\cdot ,..., \cdot}}$ be an $n$-Lie algebra, $\tau$ be a trace and $\AKMSpara{A,\AKMSbracket{\cdot ,..., \cdot}_\tau}$ the induced $(n+1)$-Lie algebra. Then,
if $\AKMSpara{A,\AKMSbracket{\cdot ,..., \cdot}}$ is nilpotent of class  $p$, we have $\AKMSpara{A,\AKMSbracket{\cdot ,..., \cdot}_\tau}$ is nilpotent of class at most  $p$.
Moreover, if there exists $u \in A$ such that $\AKMSbracket{u,x_1,...,x_n}_\tau = \AKMSbracket{x_1,...,x_n}, \forall x_1,...,x_n \in A$, then $\AKMSpara{A,\AKMSbracket{\cdot ,..., \cdot}}$ is nilpotent of class  $p$ if and only if $\AKMSpara{A,\AKMSbracket{\cdot ,..., \cdot}_\tau}$ is nilpotent of class  $p$.
\end{theorem}
\begin{proof} 
\begin{enumerate}
\item Suppose that $\AKMSpara{A,\AKMSbracket{\cdot ,..., \cdot}}$ is nilpotent of class $p\in \mathbb{N}$, then $C^p(A)=\{0\}$. By the preceding proposition, $C^p (A_\tau) \subseteq C^p(A)=\{0\}$, therefore $\AKMSpara{A,\AKMSbracket{\cdot ,..., \cdot}_\tau}$ is nilpotent of class at most $p$.
\item We suppose now that $\AKMSpara{A,\AKMSbracket{\cdot ,..., \cdot}_\tau}$ is nilpotent of class $p \in \mathbb{N}$, and that there exists $u \in A$ such that $\AKMSbracket{u,x_1,...,x_n}_\tau = \AKMSbracket{x_1,...,x_n}, \forall x_1,...,x_n \in A$, then $C^p(A_\tau)=\{0\}$. By the preceding proposition, $C^p(A) = C^p(A_\tau)=\{0\}$. Therefore $\AKMSpara{A,\AKMSbracket{\cdot ,..., \cdot}}$ is nilpotent, since $C^{p-1}(A) = C^{p-1}(A_\tau) \neq \{0\}$,  $\AKMSpara{A,\AKMSbracket{\cdot ,..., \cdot}_\tau}$ and $\AKMSpara{A,\AKMSbracket{\cdot ,..., \cdot}}$ have the same nilpotency class.
\end{enumerate}
\end{proof}


\section{Central Extension of $(n+1)$-Lie Algebras induced by $n$-Lie Algebras}\label{AKMS:Extension}
In this section, we study central extensions of $(n+1)$-Lie algebras induced by $n$-Lie algebras. We racall first some basics.
\begin{definition}
Let $A,B,C$ be three $n$-Lie algebras ($n\geq 2$). An extension of $B$ by $A$ is a short sequence: 
\[ A \overset{\lambda}{\to} C \overset{\mu}{\to} B, \]
such that $\lambda$ is an injective homomorphism, $\mu$ is a surjective homomorphism, and\\ $\operatorname{Im} \lambda \subset \ker \mu$. We say also that $C$ is an extension of $B$ by $A$.
\end{definition}
\begin{definition}
Let $A$, $B$ be two $n$-Lie algebras, and  $A \overset{\lambda}{\to} C \overset{\mu}{\to} B$ be an extension of $B$ by $A$.
\begin{itemize}
\item The extension is said to be trivial if there exists an ideal $I$ of $C$ such that\\ $C = \ker \mu \oplus I$.
\item It is said to be central if $\ker \mu \subset Z (C)$.
\end{itemize}
\end{definition}
We may equivalently define central extensions by a $1$-dimensional algebra (we will simply call it central extension) this way:
\begin{definition}
Let $A$ be an $n$-Lie algebra. We call central extension of $A$ the space $\bar{A}=A\oplus \mathbb{K} c$ equipped with the bracket:
\[ \AKMSbracket{x_1,...,x_n}_c = \AKMSbracket{x_1,...,x_n} + \omega\AKMSpara{x_1,...,x_n} c \text{ and } \AKMSbracket{x_1,...,x_{n-1},c}_c = 0,\forall x_1,...,x_n \in A, \] 
where $\omega$ is a skew-symmetric $n$-linear form such that $\AKMSbracket{\cdot,...,\cdot}_c$ satisfies the Nambu identity (or Jacobi identity for $n=2$). 
\end{definition}

\begin{proposition}[\cite{aip:review}]
\begin{enumerate}
\item The bracket of a central extension satisfies the fundamental identity (resp. Jacobi identity) if and only if $\omega$ is a $2$-cocycle for the scalar cohomology of $n$-Lie algebras (resp. Lie algebras).
\item Two central extensions of an $n$-Lie algebra (resp. Lie algebra) $A$ given by two maps $\omega_1$ and $\omega_2$ are isomorphic if and only if $\omega_2 - \omega_1$ is a $2$-coboundary for the scalar cohomology of $n$-Lie algebras (resp. Lie algebras).
\end{enumerate}
\end{proposition}

\begin{proof}
\begin{enumerate}
\item Let $(A,[\cdot,...,\cdot])$ be an $n$-Lie algebra, let $\omega$ be a skew-symmetric $n$-linear form on $A$. Define on $\bar{A} = A \oplus \mathbb{K} c$ the bracket $[\cdot,...,\cdot]_c$ by:
\[ \AKMSbracket{x_1,...,x_n}_c = \AKMSbracket{x_1,...,x_n} + \omega(x_1,...,x_n) c, \forall x_1,...,x_n \in A \]
and \[\AKMSbracket{x_1,...,x_{n-1},c}_c = 0, \forall x_1,...,x_{n-1} \in A\]

\begin{align*}
X\cdot_c (Y \cdot_c z) &- (X \cdot_c Y) \cdot_c z - Y \cdot_c (X \cdot_c z) = X \cdot_c (Y \cdot z + \omega(Y,z)c) \\
&- \AKMSpara{X \cdot Y + \sum_{i=1}^{n-1} (y_1,...,y_{i-1},\omega(X,y_i)c,...,y_{n-1})}\cdot_c z\\
&- Y \cdot_c (X \cdot z + \omega(X,z)c)\\
&= X\cdot_c (Y \cdot z) - (X \cdot Y) \cdot_c z - Y \cdot_c (X \cdot z)\\
&= X\cdot (Y \cdot z) + \omega(X,Y \cdot z) c\\
&- (X \cdot Y) \cdot z - \omega(X \cdot Y, z) c \\
&- Y \cdot_c (X \cdot z) - \omega(Y, X \cdot z) c\\
&= X\cdot (Y \cdot z)  - (X \cdot Y) \cdot z - Y \cdot_c (X \cdot z) \\
&+ \omega(X,Y \cdot z) c - \omega(X \cdot Y, z) c - \omega(Y, X \cdot z) c\\
&= d^2\omega(X,Y,z) c.
\end{align*}
That is, fundamental identity is satisfied if and only if $\omega$ is a $2$-cocycle for the scalar cohomology of $A$.
\item Let $\omega_1, \omega_2 \in Z^2(A,\mathbb{K})$ such that $\omega_2 - \omega_1 = \alpha(\AKMSbracket{\cdot,...,\cdot})$, with $\alpha \in C^1(A,\mathbb{K})$. Let $(\bar{A},[\cdot,...,\cdot]_{\omega_1})$ and $(\bar{A},[\cdot,...,\cdot]_{\omega_2})$ be two central extensions of $A$ defined by $\omega_1$ and $\omega_2$ respectively. Consider \[f : (\bar{A},[\cdot,...,\cdot]_{\omega_1}) \to (\bar{A},[\cdot,...,\cdot]_{\omega_2})\] defined by: \[f(x) = x + \alpha(P_A(x)) c,\] where $P_A$ is the projection of range $A$. We have:
\begin{align*}
f(\AKMSbracket{x_1,...,x_n}_{\omega_1}) &= \AKMSbracket{x_1,...,x_n} + \omega_1({x_1,...,x_n}) c + \alpha(\AKMSbracket{x_1,...,x_n}) c\\
&= \AKMSbracket{x_1,...,x_n} + \omega_2({x_1,...,x_n}) c\\
&= \AKMSbracket{x_1,...,x_n}_{\omega_2}\\
&= \AKMSbracket{x_1+\alpha(P_A(x_1))c,...,x_n + \alpha(P_A(x_n))c}_{\omega_2}\\
&= \AKMSbracket{f(x_1),...,f(x_n)}_{\omega_2},
\end{align*}
that means $f$ is an $n$-Lie algebras homomorphism. Let's prove now that it is an isomorphism:
\begin{align*}
\ker(f)&=\{ x \in \bar{A} : f(x)=0 \}\\
&= \{ x \in \bar{A} : x+\alpha(P_A(x))c=0 \}\\
&= \{ x \in \bar{A} : P_A(x)+(x_c+\alpha(P_A(x)))c=0 \} (x = P_A(x) + x_c c)\\
&= \{ x \in \bar{A} : P_A(x)=0 \text{ and } x_c+\alpha(P_A(x))=0 \} =\{0\},
\end{align*}
which means that $f$ is injective. Therefore one concludes, when $A$ is finite dimensional, that it is bijective. We prove now that $f$ is surjective, so the result holds in infinite dimensional case:
\begin{align*}
\operatorname{Im}(f) &= \{ f(x) : x \in \bar{A} \}\\
&= \{ x + \alpha(P_A(x))c : x \in \bar{A} \}\\
&= \{ P_A(x) + (x_c +\alpha(P_A(x))) c : x = P_A(x)+x_c c \in \bar{A} \} = \bar{A},
\end{align*}
which means that $f$ is an $n$-Lie algebras isomorphism. 
\end{enumerate}
\end{proof}

Now, we look at the question of whether a central extension of an $n$-Lie algebra may give rise to a central extension of the induced $(n+1)$-Lie algebra (by some trace $\tau$), the answer is given in the following theorem:
\begin{theorem}
Let $(A,\AKMSbracket{\cdot,...,\cdot})$ be an $n$-Lie algebra, $\tau$ be a trace and $\AKMSpara{A,\AKMSbracket{\cdot,...,\cdot}_\tau}$ be the induced $(n+1)$-Lie algebra. Let $\AKMSpara{\bar{A},\AKMSbracket{\cdot,...,\cdot}_c}$ be a central extension of $(A,\AKMSbracket{\cdot,...,\cdot})$, where 
\[\bar{A}=A\oplus \mathbb{K} c \text{ and } \AKMSbracket{x_1,...,x_n}_c = \AKMSbracket{x_1,...,x_n} + \omega\AKMSpara{x_1,...,x_n}c, \] 
and assume that $\tau$ extends  to $\bar{A}$ by $\tau(c)=0$. Then the $(n+1)$-Lie algebra $\AKMSpara{\bar{A},\AKMSbracket{\cdot,...,\cdot}_{c,\tau}}$ induced by $\AKMSpara{\bar{A},\AKMSbracket{\cdot,...,\cdot}_c}$, is a central extension of $(A,\AKMSbracket{\cdot,...,\cdot}_\tau)$, where 
\[\AKMSbracket{x_1,...,x_n}_{c,\tau} = \AKMSbracket{x_1,...,x_{n+1}}_\tau + \omega_\tau \AKMSpara{x_1,...,x_{n+1}}c\]  with \[\omega_\tau \AKMSpara{x_1,...,x_{n+1}} = \sum_{i=1}^{n+1} (-1)^{i-1} \tau\AKMSpara{x_i} \omega(x_1,...,\hat{x}_i,...,x_{n+1}). \]
\end{theorem} 
\begin{proof}
Let $x_1,...,x_{n+1} \in A$:
\begin{align*}
\AKMSbracket{x_1,...,x_{n+1}}_{c,\tau} &= \sum_{i=1}^{n+1} (-1)^{i-1} \tau\AKMSpara{x_i}\AKMSbracket{x_1,...,\hat{x}_i,...,x_{n+1}}_c \\
&=  \sum_{i=1}^{n+1} (-1)^{i-1} \tau\AKMSpara{x_i}\left( \AKMSbracket{x_1,...,\hat{x}_i,...,x_{n+1}} +  \omega(x_1,...,\hat{x}_i,...,x_{n+1})c \right) \\
&=  \sum_{i=1}^{n+1} (-1)^{i-1} \tau\AKMSpara{x_i} \AKMSbracket{x_1,...,\hat{x}_i,...,x_{n+1}}    \\
&+  \left( \sum_{i=1}^{n+1} (-1)^{i-1} \tau\AKMSpara{x_i} \omega(x_1,...,\hat{x}_i,...,x_{n+1}) \right) c \\
&= \AKMSbracket{x_1,...,x_{n+1}}_\tau + \omega_\tau \AKMSpara{x_1,...,x_{n+1}} c.
\end{align*}
The map $\omega_\tau \AKMSpara{x_1,...,x_{n+1}} = \sum_{i=1}^{n+1} (-1)^{i-1} \tau\AKMSpara{x_i} \omega(x_1,...,\hat{x}_i,...,x_{n+1})$ is a skew-symmetric $(n+1)$-linear form, and $\AKMSbracket{\cdot,...,\cdot}_{c,\tau}$ satisfies the fundamental identity. We have also:
\begin{align*}
\AKMSbracket{x_1,...,x_n,c}_{c,\tau}&= \sum_{i=1}^{n} (-1)^{i-1} \tau\AKMSpara{x_i}\AKMSbracket{x_1,...,\hat{x}_i,...,c}_c + (-1)^n \tau(c) \AKMSbracket{x_1,...,x_n} \\
&= 0. \qquad \Big(\AKMSbracket{x_{i_1},...,x_{i_{n-1}},c}_c = 0 \text{ and } \tau\AKMSpara{c} = 0.\Big)
\end{align*}
Therefore $\AKMSpara{\bar{A},\AKMSbracket{\cdot,...,\cdot}_{c,\tau}}$ is a central extension of $(A,\AKMSbracket{\cdot,...,\cdot}_\tau)$. 
\end{proof}


\section{Cohomology of $(n+1)$-Lie algebras induced by $n$-Lie algebras}\label{AKMS:Cohomology}
In this section, we study the connections between  the cohomology of a given $n$-Lie algebra and the cohomology of the induced $(n+1)$-Lie algebra.

Let $(A,\AKMSbracket{\cdot,...,\cdot})$ be an $n$-Lie algebra, $\tau$ a trace and $(A,\AKMSbracket{\cdot,...,\cdot}_\tau)$ the induced $(n+1)$-Lie algebra. Then we have the following correspondence between 1-cocycles and 2-cocycles of $(A,\AKMSbracket{\cdot,...,\cdot})$ and those of $(A,\AKMSbracket{\cdot,...,\cdot}_\tau)$.

\begin{lemma}
If $f : A \to A$ is a derivation of an $n$-Lie algebra and $\tau$ is a trace map, then $\tau \circ f$ is a trace. 
\end{lemma}
\begin{proof}
For all $x,y \in A$, we have 
\begin{align*}
\tau\AKMSpara{ f\AKMSpara{ \AKMSbracket{x_1,...,x_n} } } &= \tau\AKMSpara{ \sum_{i=1}^n \AKMSbracket{x_1,...,x_{i-1},f(x_i),x_{i+1},...,x_n } }\\
&=  \sum_{i=1}^n \tau\AKMSpara{ \AKMSbracket{x_1,...,x_{i-1},f(x_i),x_{i+1},...,x_n } } = 0.
\end{align*} 
\end{proof}

\begin{proposition}
Let $f:A \to A$ be a derivation of the $n$-Lie algebra $A$, then $f$ is a derivation of the induced $(n+1)$-Lie algebra if and only if
\[ \AKMSbracket{x_1,...,x_{n+1 }}_{\tau \circ f} =0, \forall x_1,...,x_{n+1 } \in A. \]

\end{proposition}

\begin{proof}
Let $f$ be a derivation of $A$ and $x_1,...,x_{n+1} \in A$:
\begin{align*}
f&\AKMSpara{\AKMSbracket{x_1,...,x_{n+1}}_\tau} = f\AKMSpara{\sum_{i=1}^{n+1} (-1)^{i-1} \tau(x_i)\AKMSbracket{x_1,...,\widehat{x_i},...,x_{n+1} } }\\
&= \sum_{i=1}^{n+1} (-1)^{i-1} \tau(x_i) f\AKMSpara{\AKMSbracket{x_1,...,\widehat{x_i},...,x_{n+1} } }\\
&= \sum_{i=1}^{n+1} (-1)^{i-1} \tau(x_i) \sum_{j=1 ; j\neq i}^{n+1} \AKMSbracket{x_1,...,f(x_j),...,\widehat{x_i},...,x_{n+1} }\\
&= \sum_{j=1}^{n+1}  \sum_{i=1 ; i\neq j}^{n+1} (-1)^{i-1} \tau(x_i) \AKMSbracket{x_1,...,f(x_j),...,\widehat{x_i},...,x_{n+1} } \\
&+ \sum_{j=1}^{n+1} (-1)^{j-1} \tau\AKMSpara{f(x_j)}\AKMSbracket{x_1,...,\widehat{x_j},...,x_{n+1}} - \sum_{j=1}^{n+1} (-1)^{j-1} \tau\AKMSpara{f(x_j)}\AKMSbracket{x_1,...,\widehat{x_j},...,x_{n+1}}  \\
&= \sum_{j=1}^{n+1} \sum_{i=1 ; i\neq j}^{n+1} (-1)^{i-1} \tau(x_i) \AKMSbracket{x_1,...,f(x_j),...,\widehat{x_i},...,x_{n+1} }\\ 
&+ \sum_{j=1}^{n+1} (-1)^{j-1} \tau\AKMSpara{f(x_j)}\AKMSbracket{x_1,...,\widehat{x_j},...,x_{n+1}} - \sum_{j=1}^{n+1} (-1)^{j-1} \tau\AKMSpara{f(x_j)}\AKMSbracket{x_1,...,\widehat{x_j},...,x_{n+1}}\\
&= \sum_{j=1}^{n+1} \AKMSbracket{x_1,...,x_{j-1},f(x_j),x_{j+1},...,x_{n+1}}_\tau -  \sum_{j=1}^{n+1} (-1)^{j-1} \tau\AKMSpara{f(x_j)}\AKMSbracket{x_1,...,\widehat{x_j},...,x_{n+1}} .
\end{align*}
\end{proof}
We provide in the following two examples of computations that illustrate the Proposition above. In the sequel, $(e_i)_{1\leq i \leq dim \mathfrak{g}}$ stands for the   basis of  a Lie algebra $\mathfrak{g}$ and for $x \in \mathfrak{g}$, $(x_i)_{1 \leq i \leq dim \mathfrak{g}}$ are its coordinates with respect to  this basis.
\begin{example}[\cite{akms:ternary}]
Let $\mathfrak{gl}_2(\mathbb{K})$ be the $4$-dimensional Lie algebra defined by: 
\[
 [e_1,e_2] = 2e_2 ; \qquad
[e_1,e_3] =-2e_3 ; \qquad
[e_2,e_3] =e_1. 
\]
It induces (by the trace $\tau (x) = x_4$) a $3$-Lie algebra defined with respect to the same basis by:
\[
 [e_1,e_2,e_4] = 2e_2 ; \qquad
 [e_1,e_3,e_4] =-2e_3  ; \qquad
 [e_2,e_3,e_4] =e_1. 
\]
The derivations of the Lie algebra are defined as:
\begin{align*}
f(e_1)&= -2 a_1 e_2 -2 a_2 e_3, \\
f(e_2)&= a_2 e_1 + a_3 e_2, \\
 f(e_3)&= a_1 e_1 + a_3 e_3, \\
 f(e_4)&= a_4 e_4,
\end{align*}
where $a_1,...,a_4$ are parameters. The first cohomology group $H^1(\mathfrak{gl}_2(\mathbb{K}), \mathfrak{gl}_2(\mathbb{K}))$ is one-dimensional and is spanned by $f_1$ defined as:
\[f_1(e_4)=  e_4 ; \qquad f_1(e_i) = 0 \text{ for } i \neq 4. \]
While the derivations of the corresponding $3$-Lie algebra are defined by:
\begin{align*}
g(e_1)&= a_1 e_1 - 2 a_2 e_2 - 2 a_3 e_3, \\
g(e_2)&= a_3 e_1 + a_4 e_2, \\
g(e_3)&= a_2 e_1 + (2 a_1 - a_4) e_3, \\
g(e_4)&= a_5 e_1 + a_6 e_2 + a_7 e_3 + a_1 e_4,
\end{align*}
where $a_1,...,a_7$ are parameters. The cohomology group is also one-dimensional. 
It turns out that
\[ \dim Z^1(\mathfrak{gl}_2(\mathbb{K}), \mathfrak{gl}_2(\mathbb{K})) = 4 \text{ and }
\dim Z^1(\mathfrak{gl}_2(\mathbb{K})_\tau, \mathfrak{gl}_2(\mathbb{K})_\tau) = 7 \] and
\[ \dim H^1(\mathfrak{gl}_2(\mathbb{K}), \mathfrak{gl}_2(\mathbb{K})) = \dim H^1(\mathfrak{gl}_2(\mathbb{K})_\tau, \mathfrak{gl}_2(\mathbb{K})_\tau).\]
Observe that the cohomology class generating $H^1(\mathfrak{gl}_2(\mathbb{K}), \mathfrak{gl}_2(\mathbb{K}))$ is not included in $Z^1(\mathfrak{gl}_2(\mathbb{K})_\tau, \mathfrak{gl}_2(\mathbb{K})_\tau)$. The cohomology group $H^1(\mathfrak{gl}_2(\mathbb{K})_\tau, \mathfrak{gl}_2(\mathbb{K})_\tau)$ is spanned by $g_1$ defined as:
\[ g_1(e_1) = e_1 ; \quad g_1(e_2) = 0 ; \quad g_1(e_3) = 2e_3 ; \quad g_1(e_4) = e_4. \]
\end{example}
\begin{example}[\cite{akms:ternary}]
We consider the two Lie algebras $M^4$, $M^5$ and $M^8$. See Theorem \ref{Class2Lie} for the definitions. 
For $M^4$, with the trace $\tau$ defined by $\tau(x)=x_1+x_2+x_4$, we still have the same dimensions \[\dim H^1(M^4, M^4)=\dim H^1(M^4_\tau, M^4_\tau) = 6.\]
While (with $\tau(x)=x_1$)
\[ \dim H^1(M^5,M^5) = 8, \]  \[ \dim H^1(M^5_\tau,M^5_\tau) = 9, \]
and (with $\tau(x)=x_1 + x_3$)
\[ \dim H^1(M^8,M^8) = 0 \]  \[ \dim H^1(M^8_\tau,M^8_\tau) = 4. \]
We conjecture that, for a Lie algebra $\mathfrak{g}$:
\[\dim Z^1(\mathfrak{g}, \mathfrak{g}) \leq \dim Z^1(\mathfrak{g}_\tau, \mathfrak{g}_\tau)\]
and
\[\dim H^1(\mathfrak{g}, \mathfrak{g}) \leq \dim H^1(\mathfrak{g}_\tau, \mathfrak{g}_\tau). \]

\end{example}

Now, we consider the 2-cocycles of an $(n+1)$-Lie algebra induced by an $n$-Lie algebra.

\begin{proposition}\label{AKMS-Z2ad}
Let $(A,\AKMSbracket{\cdot,...,\cdot})$ be an $n$-Lie algebra, $\tau$ be a trace and $(A,\AKMSbracket{\cdot,...,\cdot}_\tau)$ be the induced $(n+1)$-Lie algebra. Let $\varphi \in Z^2_{ad}(A,A)$ such that:
\begin{enumerate}
\item $\sum_{i=1}^n \sum_{k=1 ; k\neq i}^n (-1)^{k+n-1} \tau(y_i)\tau(y_k) \varphi(y_1,...,\widehat{y_k},...,y_{i-1}, X_n \cdot x_n, y_{i+1},...,y_n,z)$,
\item $ \sum_{i=1}^n \sum_{k=1 ; k \neq i}^n (-1)^{n+k-1} \tau(y_i)\tau(y_k) \AKMSbracket{y_1,...,\widehat{y_k},...,y_{i-1},\varphi(X_n,x_n),...,y_n,z}$,
\item $\tau\circ \varphi = 0$.
\end{enumerate}
Then $\varphi_\tau \AKMSpara{X,z} = \sum_{i=1}^n (-1)^{i-1} \tau(x_i) \varphi(X_i , z)+(-1)^n \tau(z)\varphi(X_n , x_n)$ is a 2-cocycle of the induced $(n+1)$-Lie algebra.
\end{proposition}
\begin{proof}
Let $\varphi \in Z^2_{ad}(A,A)$ such that the conditions above are satisfied, and let \[ \varphi_\tau\AKMSpara{X,z} = \sum_{i=1}^n (-1)^{i-1} \tau(x_i) \varphi(X_i , z)+(-1)^n \tau(z)\varphi(X_n , x_n).\] Then we have:

\begin{align*}
d^2& \varphi_\tau(X,Y,z) = - \varphi_\tau(X \cdot Y, z) - \varphi_\tau(Y, X\cdot z) + \varphi_\tau(X, Y \cdot z) \\
&-\AKMSpara{\varphi_\tau(X, \cdot)\cdot Y}\cdot z - Y\cdot \varphi_\tau (X,z) + X \cdot \varphi_\tau(Y,z)\\
&= - \sum_{j=1}^n \sum_{k=1}^n (-1)^{j+k}\tau(x_j)\tau(x_k) \varphi(X_j\cdot Y_k , z)\\
&- \sum_{j=1}^n \sum_{i=1}^n (-1)^{j+n-1} \tau(x_j)\tau(z) \varphi(y_1,...,y_{i-1}, X_j \cdot y_i, y_{i+1},...,y_n)\\
&- \sum_{i=1}^n \sum_{k=1 ; k\neq i}^n (-1)^{k+n-1} \tau(y_i)\tau(y_k) \varphi(y_1,...,\widehat{y_k},...,y_{i-1}, X_n \cdot x_n, y_{i+1},...,y_n,z)\\
&- \sum_{i=1}^n \tau(y_i)\tau(z) \varphi(y_1,...,y_{i-1}, X_n \cdot x_n, y_{i+1},...,y_n)\\%
&- \sum_{j=1}^n \sum_{k=1}^n (-1)^{j+k} \tau(x_j)\tau(y_k) \varphi(Y_k,X_j \cdot z) - \sum_{i=1}^n (-1)^{i+n-1} \tau(y_i)\tau(z) \varphi(Y_i, X_n \cdot x_n)\\%
&+ \sum_{j=1}^n \sum_{k=1}^n (-1)^{j+k} \tau(x_j) \tau(y_k) \varphi(X_j, Y_k \cdot z) + \sum_{j=1}^n (-1)^{j+n-1} \tau(x_j)\tau(z) \varphi(X_j,Y_n \cdot y_n)\\%
& - \sum_{j=1}^n \sum_{k=1}^n (-1)^{j+k} \tau(x_j) \tau(y_k) \AKMSpara{\varphi_\tau(X_j, \cdot)\cdot Y_k}\cdot z\\
&- \sum_{i=1}^n \sum_{j=1}^n (-1)^{n+j-1} \tau(x_j)\tau(z) \AKMSbracket{y_1,...,y_{i-1},\varphi(X_j,y_i),...,y_n} \\
&- \sum_{i=1}^n \sum_{k=1 ; k \neq i}^n (-1)^{n+k-1} \tau(y_i)\tau(y_k) \AKMSbracket{y_1,...,\widehat{y_k},...,y_{i-1},\varphi(X_n,x_n),...,y_n,z}\\
&- \sum_{i=1}^n \tau(y_i)\tau(z) \AKMSbracket{y_1,...,y_{i-1},\varphi(X_n,x_n),...,y_n}%
- \sum_{j=1}^n \sum_{k=1}^n (-1)^{j+k} \tau(x_j)\tau(y_k) Y_k\cdot \varphi(X_j, z) \\
&- \sum_{i=1}^n (-1)^{i+n-1} \tau(y_i)\tau(z) Y_i \cdot \varphi(X_n, x_n)%
+ \sum_{j=1}^n \sum_{k=1}^n (-1)^{j+k} \tau(x_j) \tau(y_k) X_j\cdot \varphi(Y_k , z)\\
&+ \sum_{j=1}^n (-1)^{j+n-1} \tau(x_j)\tau(z) X_j\cdot \varphi(Y_n, y_n) %
- \sum_{i=1}^n \sum_{j=1}^n (-1)^{i+j} \tau(x_j) \tau(\varphi(X_j,y_i)) Y_i \cdot z \\
&- \sum_{i=1}^n (-1)^{i+n-1} \tau(x_j) \tau(\varphi(X_j,y_i)) Y_i \cdot z\\
&- \sum_{i=1}^n (-1)^{n+i-1} \tau(x_i) \tau(\varphi(X_i,z) Y_n \cdot y_n -\tau(z)\tau(\varphi(X_n,x_n)) Y_n \cdot y_n \\
&+ \sum_{i=1}^n (-1)^{n+i-1} \tau(y_i) \tau(\varphi(Y_i,z) X_n \cdot x_n -\tau(z)\tau(\varphi(Y_n,y_n)) X_n \cdot x_n \\
&= \sum_{j=1}^n \sum_{k=1}^n (-1)^{j+k} \tau(x_j) \tau(x_k) d^2\varphi(X_j,Y_k,z) \\ 
& - \sum_{j=1}^n (-1)^{j+n-1} \tau(x_j)\tau(z) \AKMSpara{\sum_{i=1}^{n-1} \varphi(y_1,...,y_{i-1},X_j \cdot y_i,y_{i+1},...,y_{n-1},y_n) + \varphi(Y_n, X_j \cdot y_n)}\\
&+ \sum_{j=1}^n \tau(x_j)\tau(z) \varphi(X_j, Y_n \cdot y_n) \\
& - \sum_{j=1}^n (-1)^{j+n-1} \tau(x_j)\tau(z) \AKMSpara{\sum_{i=1}^{n-1} \AKMSbracket{y_1,...,y_{i-1},X_j \cdot y_i,y_{i+1},...,y_{n-1},y_n} + Y_n \cdot \varphi( X_j, y_n)}\\
&+ \sum_{j=1}^n \tau(x_j)\tau(z) X_j\cdot \varphi(Y_n, y_n) - \sum_{i=1}^n \tau(y_i) \tau(z) \AKMSpara{(-1)^{n-i}+(-1)^{n+i-1}}\varphi(Y_i,X_n \cdot x_n)\\
&- \sum_{i=1}^n \sum_{k=1 ; k\neq i}^n (-1)^{k+n-1} \tau(y_i)\tau(y_k) \varphi(y_1,...,\widehat{y_k},...,y_{i-1}, X_n \cdot x_n, y_{i+1},...,y_n,z)\\
&- \sum_{i=1}^n \sum_{k=1 ; k \neq i}^n (-1)^{n+k-1} \tau(y_i)\tau(y_k) \AKMSbracket{y_1,...,\widehat{y_k},...,y_{i-1},\varphi(X_n,x_n),...,y_n,z}\\
&- \sum_{i=1}^n \tau(y_i)\tau(z) \AKMSbracket{y_1,...,y_{i-1},\varphi(X_n,x_n),...,y_n}\\
&- \sum_{i=1}^n \sum_{j=1}^n (-1)^{i+j} \tau(x_j) \tau(\varphi(X_j,y_i)) Y_i \cdot z  - \sum_{i=1}^n (-1)^{i+n-1} \tau(x_j) \tau(\varphi(X_j,y_i)) Y_i \cdot z\\
&- \sum_{i=1}^n (-1)^{n+i-1} \tau(x_i) \tau(\varphi(X_i,z) Y_n \cdot y_n -\tau(z)\tau(\varphi(X_n,x_n)) Y_n \cdot y_n \\
&+ \sum_{i=1}^n (-1)^{n+i-1} \tau(y_i) \tau(\varphi(Y_i,z) X_n \cdot x_n +\tau(z)\tau(\varphi(Y_n,y_n)) X_n \cdot x_n \\
& =  \sum_{j=1}^n \sum_{k=1}^n (-1)^{j+k} \tau(x_j) \tau(x_k) d^2\varphi(X_j,Y_k,z)  + \sum_{j=1}^n (-1)^{j+n-1} \tau(x_j)\tau(z) d^2 \varphi(X_j,Y_n,y_n)\\
&- \sum_{i=1}^n \sum_{k=1 ; k\neq i}^n (-1)^{k+n-1} \tau(y_i)\tau(y_k) \varphi(y_1,...,\widehat{y_k},...,y_{i-1}, X_n \cdot x_n, y_{i+1},...,y_n,z)\\
&- \sum_{i=1}^n \sum_{k=1 ; k \neq i}^n (-1)^{n+k-1} \tau(y_i)\tau(y_k) \AKMSbracket{y_1,...,\widehat{y_k},...,y_{i-1},\varphi(X_n,x_n),...,y_n,z}\\
&- \sum_{i=1}^n \sum_{j=1}^n (-1)^{i+j} \tau(x_j) \tau(\varphi(X_j,y_i)) Y_i \cdot z - \sum_{i=1}^n (-1)^{i+n-1} \tau(x_j) \tau(\varphi(X_j,y_i)) Y_i \cdot z\\
&- \sum_{i=1}^n (-1)^{n+i-1} \tau(x_i) \tau(\varphi(X_i,z) Y_n \cdot y_n -\tau(z)\tau(\varphi(X_n,x_n)) Y_n \cdot y_n \\
&+ \sum_{i=1}^n (-1)^{n+i-1} \tau(y_i) \tau(\varphi(Y_i,z) X_n \cdot x_n +\tau(z)\tau(\varphi(Y_n,y_n)) X_n \cdot x_n \\
&= 0.
\end{align*}

\end{proof}

\begin{proposition}
Every $1$-cocycle for the scalar cohomology of an $n$-Lie algebra $(A,\AKMSbracket{\cdot,...,\cdot})$ is a $1$-cocycle for the scalar cohomology of the induced $(n+1)$-Lie algebra. Notice that $1$-cocycles for the scalar cohomology are exactly traces.
\end{proposition}
\begin{proof}
Let $\omega$ be a $1$-cocycle for the scalar cohomology of $(A,\AKMSbracket{\cdot,...,\cdot})$, then 
\[\forall x_1,...,x_{n-1},z \in A, d^1 \omega(x_1,...,x_{n-1},z) = \omega \AKMSpara{\AKMSbracket{x_1,...,x_{n-1},z}} = 0,\]
 which is equivalent to $\AKMSbracket{A,...,A} \subset \ker \omega$. By Remark \ref{AKMS-D3subD}, $\AKMSbracket{A,...,A}_\tau \subset \AKMSbracket{A,...,A}$ and then $\AKMSbracket{A,...,A}_\tau \subset \ker \omega$, that is 
\[ \forall x_1,...,x_n,z \in A, \omega\AKMSpara{\AKMSbracket{x_1,...,x_n,z}_\tau}=d^1 \omega \AKMSpara{x_1,...,x_n,z} = 0.\]
 It means that $\omega$ is a $1$-cocycle for the scalar cohomology of $\AKMSpara{A,\AKMSbracket{\cdot,...,\cdot}_\tau}$.  
\end{proof}

\begin{proposition}\label{AKMS-Z2tri}
Let $\alpha \in Z^2_{0}(A,\mathbb{K})$ such that:

\[  \sum_{i=1}^n \sum_{k=1 ; k\neq i}^n (-1)^{k+n-1} \tau(y_i)\tau(y_k) \alpha(y_1,...,\widehat{y_k},...,y_{i-1}, X_n \cdot x_n, y_{i+1},...,y_n,z)=0.\]

Then $\alpha_\tau\AKMSpara{X,z} = \sum_{i=1}^n (-1)^{i-1} \tau(x_i) \alpha(X_i , z)+(-1)^n \tau(z)\alpha(X_n , x_n)$ is a $2$-cocycle of the induced $(n+1)$-Lie algebra.
\end{proposition}
\begin{proof}
Let $\alpha \in Z^2_{0}(A,\mathbb{K})$ satisfying the condition above, and let \[\alpha_\tau\AKMSpara{X,z} = \sum_{i=1}^n (-1)^{i-1} \tau(x_i) \alpha(X_i , z)+(-1)^n \tau(z)\alpha(X_n , x_n).\] Then we have:
\begin{align*}
d^2 &\alpha_\tau(X,Y,z) = - \alpha_\tau(X \cdot Y, z) - \alpha_\tau(Y, X\cdot z) + \alpha_\tau(X, Y \cdot z) \\
&= - \sum_{j=1}^n \sum_{k=1}^n (-1)^{j+k}\tau(x_j)\tau(x_k) \alpha(X_j\cdot Y_k , z)\\
&- \sum_{j=1}^n \sum_{i=1}^n (-1)^{j+n-1} \tau(x_j)\tau(z) \alpha(y_1,...,y_{i-1}, X_j \cdot y_i, y_{i+1},...,y_n)\\
&- \sum_{i=1}^n \sum_{k=1 ; k\neq i}^n (-1)^{k+n-1} \tau(y_i)\tau(y_k) \alpha(y_1,...,\widehat{y_k},...,y_{i-1}, X_n \cdot x_n, y_{i+1},...,y_n,z)\\
&- \sum_{i=1}^n \tau(y_i)\tau(z) \alpha(y_1,...,y_{i-1}, X_n \cdot x_n, y_{i+1},...,y_n)\\
&- \sum_{j=1}^n \sum_{k=1}^n (-1)^{j+k} \tau(x_j)\tau(y_k) \alpha(Y_k,X_j \cdot z) - \sum_{i=1}^n (-1)^{i+n-1} \tau(y_i)\tau(z) \alpha(Y_i, X_n \cdot x_n)\\
&+ \sum_{j=1}^n \sum_{k=1}^n (-1)^{j+k} \tau(x_j) \tau(y_k) \alpha(X_j, Y_k \cdot z) + \sum_{j=1}^n (-1)^{j+n-1} \tau(x_j)\tau(z) \alpha(X_j,Y_n \cdot y_n)\\
&= \sum_{j=1}^n \sum_{k=1}^n (-1)^{j+k} \tau(x_j) \tau(x_k) d^2\alpha(X_j,Y_k,z) \\
& - \sum_{j=1}^n (-1)^{j+n-1} \tau(x_j)\tau(z) \AKMSpara{\sum_{i=1}^{n-1} \alpha(y_1,...,y_{i-1},X_j \cdot y_i,y_{i+1},...,y_{n-1},y_n) + \alpha(Y_n, X_j \cdot y_n)}\\
&+ \sum_{j=1}^n \tau(x_j)\tau(z) \alpha(X_j, Y_n \cdot y_n) - \sum_{i=1}^n \tau(y_i) \tau(z) \AKMSpara{(-1)^{n-i}+(-1)^{n+i-1}}\alpha(Y_i,X_n \cdot x_n)\\
&- \sum_{i=1}^n \sum_{k=1 ; k\neq i}^n (-1)^{k+n-1} \tau(y_i)\tau(y_k) \alpha(y_1,...,\widehat{y_k},...,y_{i-1}, X_n \cdot x_n, y_{i+1},...,y_n,z)\\
& =  \sum_{j=1}^n \sum_{k=1}^n (-1)^{j+k} \tau(x_j) \tau(x_k) d^2\alpha(X_j,Y_k,z) + \sum_{j=1}^n (-1)^{j+n-1} \tau(x_j)\tau(z) d^2 \alpha(X_j,Y_n,y_n)\\
&- \sum_{i=1}^n \sum_{k=1 ; k\neq i}^n (-1)^{k+n-1} \tau(y_i)\tau(y_k) \alpha(y_1,...,\widehat{y_k},...,y_{i-1}, X_n \cdot x_n, y_{i+1},...,y_n,z)\\
&= 0.
\end{align*}

\end{proof}

\begin{lemma}\label{AKMScoBtau}
Let $\alpha \in C^1(A,\mathbb{K})$. Then:
\[ d_\tau^1 \alpha\AKMSpara{x_1,...,x_{n+1}} = \sum_{i=1}^{n+1} (-1)^{i-1} \tau(x_i) d^1\alpha\AKMSpara{x_1,...,\widehat{x_i},...,x_{n+1}}, \forall x_1,...,x_{n+1} \in A, \]
where $d^p_\tau$ is the coboundary operator for the cohomology complex of $A_\tau$.
\end{lemma}
\begin{proof}
Let $\alpha \in C^1(A,\mathbb{K})$, $x_1,...,x_{n+1} \in A$, then we have:
\begin{align*}
d_\tau^1 \alpha\AKMSpara{x_1,...,x_{n+1}} &= \alpha\AKMSpara{\AKMSbracket{x_1,...,x_{n+1}}_\tau}\\
&=\sum_{i=1}^{n+1} (-1)^{i-1} \tau(x_i) \alpha\AKMSpara{\AKMSbracket{x_1,...,\widehat{x_i},...,x_{n+1}}}\\
&=  \sum_{i=1}^{n+1} (-1)^{i-1} \tau(x_i) d^1\alpha\AKMSpara{x_1,...,\widehat{x_i},...,x_{n+1}}.
\end{align*}

\end{proof}

\begin{proposition}\label{B2triv_eq}
Let $\varphi_1,\varphi_2 \in Z^2_{0}(A,\mathbb{K})$ satisfying conditions of Theorem \ref{AKMS-Z2tri}. If $\varphi_1,\varphi_2$ are in the same cohomology class then $\psi_1,\psi_2$ defined by:
\[ \psi_i\AKMSpara{x_1,...,x_{n+1}} = \sum_{j=1}^{n+1} (-1)^{j-1} \tau \AKMSpara{x_j} \varphi_i \AKMSpara{x_1,...,\widehat{x_j},...,x_{n+1}}, i=1,2, \]
are in the same cohomology class.
\end{proposition}
\begin{proof}
Let $\varphi_1,\varphi_2 \in Z^2_{0}(A,\mathbb{K})$ be two cocycles in the same cohomology class, that is \[\varphi_2 - \varphi_1 =d^1\alpha,\  \alpha \in C^1(A,\mathbb{K})\] satisfying conditions of Theorem \ref{AKMS-Z2tri}, and 
\[ \psi_i\AKMSpara{x_1,...,x_{n+1}} = \sum_{j=1}^{n+1} (-1)^{j-1} \tau \AKMSpara{x_j} \varphi_i \AKMSpara{x_1,...,\widehat{x_j},...,x_{n+1}}, i=1,2. \]
Then we have:
\begin{align*}
\psi_2\AKMSpara{x_1,...,x_{n+1}}-\psi_1\AKMSpara{x_1,...,x_{n+1}} &= \sum_{j=1}^{n+1} (-1)^{j-1} \tau \AKMSpara{x_j} (\varphi_2 \AKMSpara{x_1,...,\widehat{x_j},...,x_{n+1}} \\
&- \sum_{j=1}^{n+1} (-1)^{j-1} \tau \AKMSpara{x_j}  \varphi_1 \AKMSpara{x_1,...,\widehat{x_j},...,x_{n+1}})\\
&= \sum_{j=1}^{n+1} (-1)^{j-1} \tau \AKMSpara{x_j} (\varphi_2 -\varphi_1) \AKMSpara{x_1,...,\widehat{x_j},...,x_{n+1}}\\
&= \sum_{j=1}^{n+1} (-1)^{j-1} \tau \AKMSpara{x_j} d_1\alpha \AKMSpara{x_1,...,\widehat{x_j},...,x_{n+1}}\\
&= d_\tau^1 \alpha\AKMSpara{x_1,...,x_{n+1}}.
\end{align*}
That means that $\psi_1$ and $\psi_2$ are in the same cohomology class. 
\end{proof}

\begin{example}[\cite{akms:ternary}]
Consider the $4$-dimensional Lie algebra $(A,\AKMSbracket{.,.})$ defined  with respect to basis $\{e_1,e_2,e_3,e_4\}$ by:
\[ \AKMSbracket{e_2,e_4}=e_3\ ;\ \AKMSbracket{e_3,e_4}=e_3, \]
(remaining brackets are either obtained by skew-symmetry or zero), and let $\omega$ be a skew-symmetric bilinear form on $A$. The map $\omega$ is fully defined by the scalars 
\[\omega_{ij}=\omega\AKMSpara{e_i,e_j}, \quad 1\leq i<j \leq 4.\]
By solving the equations for $\omega$ to be a $2$-cocycle, that is 
$d ^2 \omega (e_i,e_j,e_k)=0$ for $1\leq i<j<k\leq 4$, 
we get:
\[ \omega_{13}=0 \text{ and } \omega_{23}=0. \]
Now, let $\alpha$ be a linear form on $A$, defined by $\alpha\AKMSpara{e_i}=\alpha_i, 1\leq i \leq 4$. We find that $d^1\alpha\AKMSpara{e_2,e_4}=d^1\alpha\AKMSpara{e_3,e_4}=\alpha_3$ and $d^1\alpha\AKMSpara{e_i,e_j}=0$ for other values of $i$ and $j$ ($i<j$).
Now consider the trace map $\tau$ such that $\tau\AKMSpara{e_1}=1$ and $\tau\AKMSpara{e_i}=0, i \neq 1$, and the $2$-cocycles $\lambda$ and $\mu$ defined by:
\[ \lambda\AKMSpara{e_1,e_2}=1 \]
and
\[ \mu\AKMSpara{e_2,e_4}=1\ ;\  \mu\AKMSpara{e_3,e_4}=-1. \]
Central extensions of $(A,\AKMSbracket{.,.})$ by $\lambda$ and $\mu$ are respectively given by ($\bar{A} = A \oplus \mathbb{K} c$):
\[ \AKMSbracket{e_1,e_2}_\lambda = c\ ;\ \AKMSbracket{e_2,e_4}_\lambda = e_3\ ;\  \AKMSbracket{e_3,e_4}_\lambda = e_3 \]
and 
\[ \AKMSbracket{e_2,e_4}_\mu = e_3 + c\ ;\ \AKMSbracket{e_3,e_4}_\mu = e_3 - c. \]

The $3$-Lie algebras induced by $(A,\AKMSbracket{.,.})$ and by these central extensions are given by:
\[ \AKMSbracket{e_1,e_2,e_4}_\tau = e_3\ ;\ \AKMSbracket{e_1,e_3,e_4}_\tau = e_3, \]
\[ \AKMSbracket{e_1,e_2,e_4}_{\tau,\lambda} = e_3\ ;\ \AKMSbracket{e_1,e_3,e_4}_{\tau,\lambda} = e_3, \] and
\[ \AKMSbracket{e_1,e_2,e_4}_{\tau,\mu} = e_3 + c\ ;\ \AKMSbracket{e_1,e_3,e_4}_{\tau,\mu} = e_3 - c. \]
We can see that here the central extension given by $\lambda$ induces a trivial one, while the one given by $\mu$ induces a non-trivial one. This example shows also that the converse of Proposition \ref{B2triv_eq} is, in general, not true.
\end{example}

\section{On classification of $(n+1)$-Lie Algebras Induced by $n$-Lie Algebras}\label{AKMS:Classifications}
In this section, we give a list of all $(n+1)$-Lie algebras induced by $n$-Lie algebras in dimension $d\leq n+2$, based on the classifications given in \cite{Filippov:nLie} and \cite{Bai:nLie:n+2}. It generalizes the results obtained in \cite{akms:ternary}, which part of them were found independently in \cite{Bai:rlz3}.  To this end, we shall use the following result:
\begin{proposition} \label{AKMS-3to2}
Let $\AKMSpara{A,\AKMSbracket{\cdot,...,\cdot}}$ be an $(n+1)$-Lie algebra, and $\AKMSpara{e_i}_{1\leq i \leq d}$ a basis of $A$. If there exists $e_{i_0}$ in this basis, such that the multiplication table of $\AKMSpara{A,\AKMSbracket{\cdot,...,\cdot}}$ is given by:
\[ \AKMSbracket{e_{i_0},e_{j_1},...,e_{j_n}} = x_{{j_1}...{j_n}} \in A, \  j_k \neq i_0\ (1\leq k \leq n) \text{ and } j_k : 1\leq k \leq n \text{ are all different}, \]
with $e_{i_0}$ and $x_{{j_1}...{j_n}}$ linearly independent, then the $(n+1)$-Lie algebra $\AKMSpara{A,\AKMSbracket{\cdot,...,\cdot}}$ is induced by an $n$-Lie algebra.
\end{proposition}
\begin{proof}
We define an $n$-linear skew-symmetric map $\AKMSbracket{\cdot,...,\cdot}_0$ on $A$ and a form $\tau : A \to \mathbb{K}$ by:
\[ \AKMSbracket{e_{j_1},...,e_{j_n}}_0=x_{{j_1}...{j_n}}\in A, j_k\neq i_0\  (1\leq k \leq n) ; \  j_k : 1\leq k \leq n \text{ are all different,}\] \[ \AKMSbracket{e_{i_0},e_{j_1},...,e_{j_{n-1}}}_0 = 0 \]
and
\[ \tau(x) = \tau\AKMSpara{\sum_{k=0}^d x_k e_k} = x_{i_0}. \]
The bracket $\AKMSbracket{\cdot,...,\cdot}_0$ satisfies the fundamental identity:
\begin{align*}
\AKMSbracket{x_1,...,x_{n-1},\AKMSbracket{y_1,...,y_n}}_0 &= \AKMSbracket{e_{i_0},x_1,...,x_{n-1},\AKMSbracket{e_{i_0},y_1,...,y_n}}\\
&= \AKMSbracket{\AKMSbracket{e_{i_0},x_1,...,x_{n-1},e_{i_0}},y_1,...,y_n} \\
&+ \sum_{k=1}^n  \AKMSbracket{e_{i_0},y_1,...,y_{i-1},\AKMSbracket{e_{i_0},x_1,...,x_{n-1},y_i},y_{i+1},...,y_n}\\
&= \sum_{k=1}^n  \AKMSbracket{y_1,...,y_{i-1},\AKMSbracket{x_1,...,x_{n-1},y_i},y_{i+1},...,y_n}_0.
\end{align*}
The obtained $n$-Lie bracket $\AKMSbracket{\cdot,...,\cdot}$ and the trace $\tau$ given above indeed induce the $(n+1)$-Lie bracket:
\begin{align*}
\AKMSbracket{e_{i_0},e_{j_1},...,e_{j_n}}_{0,\tau} &= \tau(e_{i_0})\AKMSbracket{e_{j_1},...,e_{j_n}}_0 + \sum_{k=1}^n (-1)^k \tau(e_{j_k})\AKMSbracket{e_{i_0},e_{j_1},...,e_{j_{k-1}},e_{j_{k+1}},...,e_{j_n}}_0 \\
&= \tau(e_{i_0})\AKMSbracket{e_{j_1},...,e_{j_n}}_0 \\
&= x_{{j_1}...{j_n}}\\
&= \AKMSbracket{e_{i_0},e_{j_1},...,e_{j_n}}.
\end{align*}
For $i \neq i_0$:
\begin{align*}
 \AKMSbracket{e_i,e_{j_1},...,e_{j_n}}_{0,\tau} &= \tau(e_i)\AKMSbracket{e_{j_1},...,e_{j_n}} + \sum_{k=1}^n (-1)^k \tau(e_{j_k})\AKMSbracket{ e_i, e_{j_1},...,e_{j_{i-1}},e_{j_{i+1}},...,e_{j_n} }\\
& = 0 = \AKMSbracket{e_i,e_{j_1},...,e_{j_n}} .
\end{align*}

\end{proof}

\begin{theorem}[\cite{Filippov:nLie} $n$-Lie algebras of dimension less than or equal to $n+1$] \label{AKMSd4}
Any $n$-Lie algebra $A$ of dimension less than or equal to $n+1$ is isomorphic to one of the following $n$-ary algebras: (omitted brackets are either obtained by skew-symmetry or  $0$, we set $\AKMSpara{e_i}_{1 \leq i \leq dim A}$ to be  a basis of $A$)
\begin{enumerate}
\item If $dim A < n$ then $A$ is abelian.
\item If $dim A = n$, then we have 2 cases:
\begin{enumerate}
\item $A$ is abelian.
\item $\AKMSbracket{e_1,...,e_n} = e_1.$
\end{enumerate}
\item if $dim A = n+1$ then we have the following cases:
\begin{enumerate}
\item $A$ is abelian.
\item $\AKMSbracket{e_2,...,e_{n+1}} = e_1$.
\item $\AKMSbracket{e_1,...,e_n}=e_1$.
\item $\AKMSbracket{e_1,...,e_{n-1},e_{n+1}} = a e_n + b e_{n+1} ; \AKMSbracket{e_1,...,e_n} = c e_n + d e_{n+1}$, with $C=
\begin{pmatrix}
a & b \\
c & d
\end{pmatrix}$ an invertible matrix. Two such algebras, defined by matrices $C_1$ and $C_2$, are isomorphic if and only if there exists a scalar $\alpha$ and an invertible matrix $B$ such that $C_2 = \alpha B C_1 B^{-1}$.
\item $\AKMSbracket{e_1,...,\widehat{e_i},...,e_n}= a_i e_i$ for $ 1 \leq i \leq r$, $2<r= \operatorname{dim} D^1(A)\leq n$, $a_i \neq 0$
\item  $\AKMSbracket{e_1,...,\widehat{e_i},...,e_n}= a_i e_i$ for $ 1 \leq i \leq n+1$ which is simple.
\end{enumerate}

\end{enumerate}
\end{theorem}

\begin{theorem}[\cite{Bai:nLie:n+2} $(n+2)$-dimensional $n$-Lie algebras] \label{AKMSd5}
Let $\mathbb{K}$ be an algebraically closed field. Any $(n+2)$-dimensional $n$-Lie algebra $A$  is isomorphic to one of the $n$-ary algebras listed below, where $A^1$ denotes $\AKMSbracket{A,...,A}$ and  $\left(e_i\right)_{1\leq i \leq n+2}$ being a basis:

\begin{enumerate}
\item If $dim A^1 = 0$ then $A$ is abelian.
\item If $dim A^1 = 1$, let $A^1=\langle e_1 \rangle$, then we have 
\begin{enumerate}
\item $A^1 \subseteq Z(A)$ : $\AKMSbracket{e_2,...,e_{n+1}}=e_1$.
\item $A^1 \nsubseteq Z(A)$ : $\AKMSbracket{e_1,...,e_n}=e_1$.
\end{enumerate}
\item If $dim A^1 = 2$, let $A^1 = \langle e_1,e_2 \rangle$, then we have 
\begin{enumerate}
\item $\AKMSbracket{e_2,...,e_{n+1}}=e_1 ; \AKMSbracket{e_3,...,e_{n+2}}=e_2$.
\item $\AKMSbracket{e_2,...,e_{n+1}}=e_1 ; \AKMSbracket{e_2,e_4,...,e_{n+2}}=e_2 ; \AKMSbracket{e_1,e_4,...,e_{n+2}}=e_1$.
\item $\AKMSbracket{e_2,...,e_{n+1}}=e_1 ; \AKMSbracket{e_1,e_3,...,e_{n+1}}=e_2$.
\item $\AKMSbracket{e_2,...,e_{n+1}}=e_1 ; \AKMSbracket{e_1,e_3,...,e_{n+1}}=e_2 ; \AKMSbracket{e_2,e_4,...,e_{n+2}}=e_2 ;\\ \AKMSbracket{e_1,e_4,...,e_{n+2}}=e_1$.
\item $\AKMSbracket{e_2,...,e_{n+1}}= \alpha e_1 + e_2 ; \AKMSbracket{e_1,e_3,...,e_{n+1}}=e_2$.
\item $\AKMSbracket{e_2,...,e_{n+1}}=\alpha e_1+e_2 ; \AKMSbracket{e_1,e_3,...,e_{n+1}}=e_2 ; \AKMSbracket{e_2,e_4,...,e_{n+2}}=e_2 ;\\ \AKMSbracket{e_1,e_4,...,e_{n+2}}=e_1$.
\item $\AKMSbracket{e_1,e_3,...,e_{n+1}}=e_1 ; \AKMSbracket{e_2,e_3,...,e_{n+1}}=e_2$.
\end{enumerate}
where $\alpha \in \mathbb{K} \setminus \left\{0\right\}$
\item If $dim A^1 = 3$, let $A^1=\langle e_1,e_2,e_3 \rangle$, then we have 
\begin{enumerate}
\item $\AKMSbracket{e_2,...,e_{n+1}}=e_1 ; \AKMSbracket{e_2,e_4,...,e_{n+2}}=-e_2 ; \AKMSbracket{e_3,...,e_{n+2}}=e_3$.
\item $\AKMSbracket{e_2,...,e_{n+1}}=e_1 ; \AKMSbracket{e_3,...,e_{n+2}}=e_3 + \alpha e_2 ; \AKMSbracket{e_2,e_4,...,e_{n+2}}=e_3 ; \AKMSbracket{e_1,e_4,...,e_{n+2}}=e_1$.
\item $\AKMSbracket{e_2,...,e_{n+1}}=e_1 ; \AKMSbracket{e_3,...,e_{n+2}}=e_3 ; \AKMSbracket{e_2,e_4,...,e_{n+2}}=e_2 ; \AKMSbracket{e_1,e_4,...,e_{n+2}}=2e_1$.
\item $\AKMSbracket{e_2,...,e_{n+1}}=e_1 ; \AKMSbracket{e_1,e_3,...,e_{n+1}}=e_2 ; \AKMSbracket{e_1,e_2,e_4,...,e_{n+1}}=e_3$.
\item $\AKMSbracket{e_1,e_4,...,e_{n+2}}=e_1 ; \AKMSbracket{e_2,e_4,...,e_{n+2}}=e_3 ; \AKMSbracket{e_3,...,e_{n+2}}=\beta e_2+(1+\beta)e_3$,\\ $\beta \in \mathbb{K} \setminus \left\{0,1\right\}$.
\item $\AKMSbracket{e_1,e_4,...,e_{n+2}}=e_1 ; \AKMSbracket{e_2,e_4,...,e_{n+2}}=e_2 ; \AKMSbracket{e_3,...,e_{n+2}}=e_3$.
\item $\AKMSbracket{e_1,e_4,...,e_{n+2}}=e_2 ; \AKMSbracket{e_2,e_4,...,e_{n+2}}=e_3 ; \AKMSbracket{e_3,...,e_{n+2}}=s e_1 + t e_2 + u e_3$. And $n$-Lie algebras corresponding to this case with coefficients $s,t,u$ and $s',t',u'$ are isomorphic if and only if there exists a non-zero element $r \in K$ such that 
\[ s=r^3 s' ; t=r^2 t' ; u=r u' .\]
\end{enumerate}
\item If $dim A^1 = r$ with $4 \leq r \leq n+1$, let $A^1 = \langle e_1,e_2,...,e_r \rangle$, then we have 
\begin{enumerate}
\item $\AKMSbracket{e_2,...,e_{n+1}}=e_1 ; \AKMSbracket{e_3,...,e_{n+2}}=e_2 ; ... ;  \AKMSbracket{e_2,...,e_{i-1},e_{i+1},...,e_{n+2}}=e_i ; \\ \AKMSbracket{e_2,...,e_{r-1},e_{r+1},...,e_{n+2}}=e_r$.
\item $\AKMSbracket{e_2,...,e_{n+1}}=e_1 ; ... ; \AKMSbracket{e_1,...,e_{i-1},e_{i+1},e_{n+1}}=e_i ; ... ; \\  \AKMSbracket{e_1,...,e_{r-1},e_{r+1},e_{n+1}}=e_r$.
\end{enumerate}
\end{enumerate}
\end{theorem}
The $n$-Lie algebras which are induced by $(n-1)$-Lie algebras are described in the following proposition:
\begin{proposition}
Let $\mathbb{K}$ be an algebraically closed field of characteristic $0$. According to Theorems \ref{AKMSd4} and \ref{AKMSd5}, the $n$-Lie algebras induced by $(n-1)$-Lie algebras of dimension $d \leq n+2$ are:
\begin{itemize}
\item If $d=n$, Theorem \ref{AKMSd4} (2.)
\item If $d=n+1$,  Theorem \ref{AKMSd4} (3.): a,b,c,d,e.
\item If $d=n+2$, Theorem \ref{AKMSd5}: (1.) (2.) (3.) (4.) and (5.) for $r \leq n$.
\end{itemize}
\end{proposition}
\begin{proof}
By applying Proposition \ref{AKMS-3to2}, it turns out that the $n$-Lie algebras given in Theorem \ref{AKMSd4} (2.) and (3.) a,b,c,d,e and Theorem \ref{AKMSd5} (1.), (2.), (3.), (4.) and (5.) for $r\leq n$ are induced by $(n-1)$-Lie algebras. The remaining algebras have derived algebras which are not abelian then they cannot be induced by $(n-1)$-Lie algebras (Theorem \ref{AKMSsolv2}). 
\end{proof}

We list, below, all $3$-dimensional and solvable $4$-dimensional Lie algebras and all $3$-Lie algebras induced by them. The $3$-dimensional Lie algebras are given in \cite{Patera_et_al:invariants} and $4$-dimensional ones were provided  partially  in \cite{degraaf4solvable}. For every 3-Lie algebra in Theorem \ref{AKMSd4} (case $n=3$) that can be induced by a Lie algebra, we provide examples of Lie algebra inducing it.
\begin{theorem}[\cite{Patera_et_al:invariants} $3$-dimensional Lie algebras]
Let $\mathfrak{g}$ be a Lie algebra and $\left(e_i\right)_{1 \leq i \leq 3}$ be a basis of $\mathfrak{g}$, then $\mathfrak{g}$ is isomorphic to one of the following algebras (remaining brackets are either obtained by skew-symmetry or zero)
\begin{enumerate}
\item abelian Lie algebra $\AKMSbracket{x,y}=0, \forall x,y \in \mathfrak{g}$.
\item $ L(3,-1): \AKMSbracket{e_1,e_2}=e_1$.
\item $ L(3,1): \AKMSbracket{e_1,e_2}=e_3$.
\item $ L(3,2,a): \AKMSbracket{e_1,e_3}=e_1 ; \AKMSbracket{e_2,e_3}=a e_2 ; 0 < |a| \leq 1$.
\item $ L(3,3): \AKMSbracket{e_1,e_3}=e_1 ; \AKMSbracket{e_2,e_3}=e_1+e_2$.
\item $ L(3,4,a): \AKMSbracket{e_1,e_3}= a e_1-e_2 ; \AKMSbracket{e_2,e_3} = e_1+a e_2 ; a\geq 0$.
\item $ L(3,5): \AKMSbracket{e_1,e_2}=e_1 ; \AKMSbracket{e_1,e_3}=-2e_2 ; \AKMSbracket{e_2,e_3}=e_3$. 
\item $ L(3,6): \AKMSbracket{e_1,e_2}=e_3 ; \AKMSbracket{e_1,e_3}=-e_2 ; \AKMSbracket{e_2,e_3}=e_1$.
\end{enumerate}
\end{theorem}

\begin{remark}
The classification given above is for the ground field $\mathbb{K} = \mathbb{R}$, if $\mathbb{K} = \mathbb{C}$ then $L\AKMSpara{3,2,\frac{x-i}{x+i}}$ is isomorphic to $L(3,4,x)$ and $L(3,5)$ is isomorphic to $L(3,6)$.
\end{remark}

\begin{theorem}[\cite{degraaf4solvable} Solvable $4$-dimensional Lie algebras]\label{Class2Lie}
Let $\mathfrak{g}$ be a solvable Lie algebra, and $\left(e_i\right)_{1 \leq i \leq 4}$ be a basis of $\mathfrak{g}$, then $\mathfrak{g}$ is isomorphic to one of the following algebras: (remaining brackets are either obtained by skew-symmetry or zero)

\begin{enumerate}
\item The abelian Lie algebra $\AKMSbracket{x,y}=0, \forall x,y \in \mathfrak{g}$.
\item $M^2$: $\AKMSbracket{e_1,e_4}=e_1 ; \AKMSbracket{e_2,e_4}=e_2 ; \AKMSbracket{e_3,e_4}=e_3$.
\item $M^3_a$ : $\AKMSbracket{e_1,e_4}=e_1 ; \AKMSbracket{e_2,e_4}=e_3 ; \AKMSbracket{e_3,e_4}=-a e_2 + (a+1)e_3$.
\item $M^4$: $\AKMSbracket{e_2,e_4}=e_3 ; \AKMSbracket{e_3,e_4}=e_3$.
\item $M^5$: $\AKMSbracket{e_2,e_4}=e_3$.
\item $M^6_{a,b}$ : $\AKMSbracket{e_1,e_4}=e_2 ; \AKMSbracket{e_2,e_4}=e_3 ; \AKMSbracket{e_3,e_4}= ae_1+be_2+e_3$.
\item $M^7_{a,b}$: $\AKMSbracket{e_1,e_4}=e_2 ; \AKMSbracket{e_2,e_4}=e_3 ; \AKMSbracket{e_3,e_4}= ae_1+be_2$ ($a=b \neq 0$ or $a=0$ or $b=0$).
\item $M^8$ : $\AKMSbracket{e_1,e_2}=e_2 ; \AKMSbracket{e_3,e_4}=e_4$.
\item $M^9_a$ : $\AKMSbracket{e_1,e_4}=e_1+ae_2 ; \AKMSbracket{e_2,e_4}=e_1 ; \AKMSbracket{e_1,e_3}=e_1 ; \AKMSbracket{e_2,e_3}=e_2$ ($X^2-X-a$ has no root in $\mathbb{K}$).
\item $M^{11}$: $\AKMSbracket{e_1,e_4}=e_1 ; \AKMSbracket{e_3,e_4}=e_3 ; \AKMSbracket{e_1,e_3}=e_2$.
\item $M^{12}$: $\AKMSbracket{e_1,e_4}=e_1 ; \AKMSbracket{e_2,e_4}=e_2 ; \AKMSbracket{e_3,e_4}=e_3 ; \AKMSbracket{e_1,e_3}=e_2$.
\item $M^{13}_a$: $\AKMSbracket{e_1,e_4}=e_1+ae_3 ; \AKMSbracket{e_2,e_4}=e_2 ; \AKMSbracket{e_3,e_4}=e_1 ; \AKMSbracket{e_1,e_3}=e_2 $.
\item $M^{14}_a$: $\AKMSbracket{e_1,e_4}=ae_3 ; \AKMSbracket{e_3,e_4}=e_1 ; \AKMSbracket{e_1,e_3}=e_2$. ($M^{14}_a$ is isomorphic to $M^{14}_b$ if and only if $a=\alpha^2 b$ for some $\alpha \neq 0$).
\end{enumerate}
\end{theorem}

In the following, for each $3$-Lie algebra (Theorem \ref{AKMSd4}), except the simple one which cannot be induced by a Lie algebra, we give examples of Lie algebras, from the above lists, that induce it: 
\begin{itemize}
\item The only non abelian $3$-dimensional $3$-Lie algebra is defined by \[ [e_1,e_2,e_3]=e_1,\] it is induced by the Lie algebra $L(3,-1)$.
\item Non abelian $4$-dimensional $3$-Lie algebras which can be induced by Lie algebras are listed in the table below, each one with Lie algebras inducing it:

\begin{tabular}{cc} \hline
\textbf{$3$-Lie algebra} & \textbf{Lie algebras inducing it} \\
\hline
$[e_1,e_2,e_3]=e_1$ & $M^3, M^4$ \\ \hline
$[e_2,e_3,e_4] = e_1$ & $M^5, M^{12}, M^{13}_a (a \neq 0), M^{14}_b$ \\ \hline
$\left.\begin{array}[c]{c}
[e_1,e_2,e_4] = a_{11} e_3 + a_{12} e_4 \\
\left[e_1,e_2,e_3\right] = a_{21} e_3 + a_{22} e_4
\end{array}\right.$ 
& $M^6_{0,b}, M^7_{0,b}, M^8, M^9_a, M^{11}, M^{13}_0$ \\ \hline

$\begin{array}{c}\AKMSbracket{e_2,e_3,e_4}= e_1 \\ \AKMSbracket{e_1,e_3,e_4} = e_2 \\ \AKMSbracket{e_1,e_2,e_4} = e_3 \end{array}$ & $\mathfrak{gl}_2(\mathbb{K})$ \\ \hline
\end{tabular}
\end{itemize}


\end{document}